\documentclass[12pt]{amsart}
\usepackage{amscd}
\usepackage{amsmath}
\usepackage{amssymb}
\usepackage{amsthm}
\usepackage{hyperref}
\usepackage{tikz}
\tikzstyle{normal}=[circle,inner sep=0pt, fill=black,  minimum size=1.5mm, draw]

\newtheorem{theorem}{Theorem}[section]
\newtheorem{lemma}[theorem]{Lemma}
\newtheorem{proposition}[theorem]{Proposition}

\theoremstyle{definition}

\newtheorem{example}[theorem]{Example}

\theoremstyle{remark}

\numberwithin{equation}{section}

\DeclareMathOperator{\supp}{supp}
\def\ab{{\mathbf a}}

\def\eb{{\mathbf e}}

\newcommand{\Fc}{\mathcal{F}}
\newcommand{\Pc}{\mathcal{P}}
\newcommand{\Hc}{\mathcal{H}}
\newcommand{\set}[1]{\left\{\,#1\,\right\}} 
\newcommand{\with}{\ \vrule\ }  

\newcommand{\NN}{\mathbb N}
\newcommand{\RR}{\mathbb R}
\newcommand{\ZZ}{\mathbb Z}

\begin{document}
\title{Edge rings satisfying Serre's condition $(R_{1})$}
\author[]{Takayuki Hibi}
\address[]{Department of Pure and Applied Mathematics,
Graduate School of Information Science and Technology,
Osaka University,
Toyonaka, Osaka 560-0043, Japan}
\email{hibi@math.sci.osaka-u.ac.jp}
\author[]{Lukas Katth\"an}
\address[]{Fachbereich Mathematik und Informatik, Philipps-Universit\"at Marburg,
35032 Marburg, Germany}%
\email{katthaen@mathematik.uni-marburg.de}%
\thanks{
\hspace{5.3mm}{\bf Key words and phrases}:
finite graph, edge ring, edge polytope, Serre's condition $(R_{1})$. \\
\hspace{5.3mm}The first author is supported by the JST CREST 
``Harmony of Gr\"obner Bases and the Modern Industrial Society.''\\
\hspace{5.3mm}This research was performed while the second author was staying
at Department of Pure and Applied Mathematics, 
Osaka University, November 2011 -- April 2012, supported by the DAAD}
\begin{abstract}
A combinatorial criterion for the edge ring of a finite connected graph
to satisfy Serre's condition $(R_{1})$ is studied.
\end{abstract}
\maketitle
\section*{Introduction}
The edge polytopes and edge rings of finite connected graphs have been 
studied from the viewpoints of both combinatorics and computational commutative algebra 
(\cite{OH98}, \cite{OH99}).
Especially, a combinatorial characterization for the edge ring to be normal 
is obtained by both \cite{OH98} and \cite{SVV} independently.
It follows immediately from \cite[Theorem 6.4.2]{BH} that
a normal edge ring is Cohen--Macaulay.
However, in general it seems unclear when the edge ring is Cohen--Macaulay.
Recall that a noetherian ring is normal if and only if 
it satisfies Serre's conditions $(R_{1})$ and $(S_{2})$.
Thus in particular an edge ring satisfying Serre's condition $(R_{1})$
is normal if and only if it is Cohen--Macaulay.
In the present paper the problem when a given edge ring satisfies Serre's condition
$(R_{1})$ is investigated.

\section{Edge rings and edge polytopes of finite connected graphs}
First, we recall from \cite{OH98} what edge rings and edge polytopes 
of finite connected graphs are.
Let $G$ be a finite connected graph on the vertex set $[\,d\,] = \{ 1, \ldots, d \}$
with $E(G) = \{ e_{1}, \ldots, e_{n}\}$ its edge set.  We always assume that
$G$ is simple, i.e., $G$ has no loop and no multiple edge.
Let $\eb_{1}, \ldots, \eb_{d}$ denote the $i$th unit coordinate vectors of $\RR^{d}$. 
We associate each edge $e = \{ i, j \} \in E(G)$ with the vector 
$\rho(e) = \eb_{i} + \eb_{j} \in \RR^{d}$.
The {\em edge polytope} is the convex polytope $\Pc_{G} \subset \RR^{d}$ which
is the convex hull of the finite set $\{ \rho(e_{1}), \ldots, \rho(e_{n}) \}$. 
Let $K[{\bf t}] = K[t_{1}, \ldots, t_{d}]$ be the polynomial ring 
in $d$ variables over a field $K$.
We associate each edge $e = \{ i, j \} \in E(G)$ with the quadratic monomial 
${\bf t}^{e} = t_{i}t_{j} \in K[{\bf t}]$.
The {\em edge ring} is the affine semigroup ring 
$K[G] = K[{\bf t}^{e_{1}}, \ldots, {\bf t}^{e_{n}}]$.

Let, in general, $\Pc \subset \RR^{d}$ be an integral convex polytope, i.e.,
a convex polytope all of whose vertices have integer coordinates, which lies
on a hyperplane $\Hc \subset \RR^{d}$ with ${\bf 0} \not\in \Hc$, where
${\bf 0}$ is the origin of $\RR^{d}$.  
We assume that $\Pc \subset \RR_{\geq 0}^{d}$, where $\RR_{\geq 0}$
is the set of nonnegative real numbers.
Then for each integer point $\ab = (a_{1}, \ldots, a_{d})$ belonging to $\Pc$,
we associate the monomial ${\bf t}^{\ab} = t_{1}^{a_{1}} \cdots t_{d}^{a_{d}}
\in K[{\bf t}]$.  The {\em toric ring} of $\Pc$ is the affine semigroup ring
$K[\Pc] = K[\{ {\bf t}^{\ab} \, : \, \ab \in \Pc \cap \ZZ^{d} \}]$. 
Thus in particular the edge ring $K[G]$ of a finite connected graph $G$ 
is the toric ring of the edge polytope $\Pc_{G}$ of $G$. 

We say that an integral convex polytope $\Pc$ is {\em normal} if its toric ring
$K[\Pc]$ is normal.  It is shown \cite{OH98} and \cite{SVV} that the edge ring of
a finite connected graph is normal if and only if $G$ satisfies the \emph{odd cycle condition}.
This condition states that for every two disjoint minimal odd cycles in $G$, there is an
 edge with one endpoint in each cycle \cite[p. 410]{OH98}. 
Recall that a toric ring is normal if and only if 
it satisfies Serre's condition $(R_{1})$ and $(S_{2})$
and that every normal toric ring is Cohen--Macaulay.
Thus in particular a toric ring satisfying Serre's condition $(R_{1})$
is normal if and only if it is Cohen--Macaulay.
In the present paper the problem when a given edge ring satisfies Serre's condition
$(R_{1})$ is investigated.

\section{When does an edge ring satisfy Serre's condition $(R_{1})$\,?}
Let $G$ be a finite connected graph on the vertex set 
$[\,d\,] = \set{1,\dotsc, d}$.  If $G$ is bipartite, then
$K[G]$ is normal and satisfies Serre's condition $(R_{1})$.
Thus in what follows we assume that $G$ is nonbipartite, i.e., 
$G$ possesses at least one odd cycle.

If $T$ is a nonempty subset of $[\,d\,]$, then the induced subgraph of $G$ on $T$ is denoted by $G_{T}$.
A nonempty subset $T$ of $[\,d\,]$ is called {\em independent} if $\{ i, j \} \not\in E(G)$
for all $i, j \in T$ with $i \neq j$.
If $T$ is independent and if $N(G;T)$ is the set of vertices
$j \in [\,d\,]$ with $\{ i, j \} \in E(G)$ for some $i \in T$, then
the {\em bipartite graph induced by $T$} 
is defined to be the bipartite graph having the vertex set $T \cup N(G;T)$
and consisting of all edges $\{ i , j \} \in E(G)$ with $i \in T$ and $j \in N(G;T)$.
We say that a nonempty subset $T\subset [\,d\,]$ is {\em fundamental} if 
\begin{itemize}
	\item 
		$T$ is independent;
	\item
		the bipartite graph induced by $T$ 
		is connected;
	\item
		either $T \cup N(G;T) = [\,d\,]$ or every connected component of
		the induced subgraph $G_{[\,d\,]\setminus (T \cup N(G;T))}$ has at least one odd cycle.
\end{itemize}
Moreover, we call a vertex $i \in [\,d\,]$ {\em regular} if every connected component of
$G_{[\,d\,]\setminus i}$ has at least one odd cycle. Note that a regular vertex is not the same as a
fundamental set with one element.

We are now in the position to state our criterion for an edge ring to satisfy
Serre's condition $(R_{1})$.
\begin{theorem}
\label{main}
	Let $G$ be a finite connected nonbipartite graph on $[\,d\,]$.
	Then the edge ring $K[G]$
	of $G$ satisfies Serre's condition $(R_1)$ if and only if the following conditions are satisfied:
	\begin{enumerate}
	\item[(\hspace{0.5mm}i\hspace{0.5mm})] 
		For every regular vertex $i\in [\,d\,]$, the induced subgraph $G_{[\,d\,]\setminus i}$ is connected;
	\item[(ii)] For every fundamental set $T \subset [\,d\,]$,
		one has either $T\cup N(G;T) = [\,d\,]$ or the induced subgraph $G_{[\,d\,]\setminus (T\cup N(G;T))}$ 
		is connected.
	\end{enumerate}
\end{theorem}

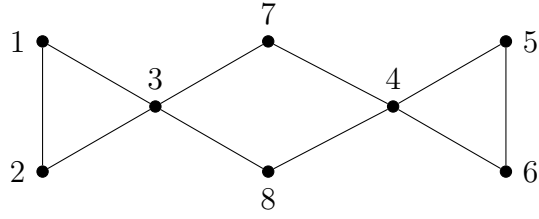
\begin{figure}
\begin{tikzpicture}[scale=1]
	\path (-2.58,0)
		node[normal, label=left:$1$]  (1) at +(120:1) {}
		node[normal, label=left:$2$]  (2) at +(240:1) {}
		node[normal, label=above:$3$] (3) at +(  0:1) {};

	\path (2.58,0)
		node[normal, label=above:$4$] (4) at +(180:1) {}
		node[normal, label=right:$5$] (5) at +( 60:1) {}
		node[normal, label=right:$6$] (6) at +(300:1) {};	
	
	\path (-0.58,0)
		node[normal, label=above:$7$]   (7) at +(60:1) {}
		node[normal, label=below:$8$] (8) at +(300:1) {}; 

	\draw (1) -- (2) -- (3) -- (1);
	\draw (4) -- (5) -- (6) -- (4);
	\draw (3) -- (7) -- (4);
	\draw (3) -- (8) -- (4);
\end{tikzpicture}
\caption{The graph $G$ of Example 1}
\label{fig:Gkgraph}
\end{figure}

\begin{example}
Let $G$ be the finite connected graph on $\{ 1, \ldots, 8 \}$ depicted in Figure~1. 
The graph $G$ clearly violates the odd cycle condition, hence the edge ring $K[G]$
is not normal.
The only vertices whose removal makes $G$ disconnected are $3$ and $4$, but both are not regular.
If $T \subset [\,8\,]$ is a set such that $G_{[\,d\,]\setminus (T\cup N(G;T))}$ is disconnected, then either $3$ or
$4$ are contained in $T\cup N(G;T)$. But then $G_{[\,d\,]\setminus (T\cup N(G;T))}$ has only one odd cycle left,
so $T$ cannot be fundamental.
Hence $K[G]$ satisfies Serre's condition $(R_{1})$.
More generally, the same argument shows that the graphs $G_{k+6}$ constructed in \cite{HHKO} satisfy
$(R_{1})$ if and only if $k \geq 2$.
\end{example}


\section{Proof of Theorem \ref{main}}
First recall the description of the facets of $\Pc_{G}$.
To every regular vertex $i$ we associate the linear form $\sigma_i: \RR^d \rightarrow \RR$ 
which projects onto the $i$-th component. 
Moreover, we set $\Hc_i = \set{x\in \RR^d \, : \, \sigma_i(x)=0}$ 
and $\Fc_i = \Pc_G \cap \Hc_i$.
Similarly, to every fundamental set $T$ we associate the linear form 
\[\sigma_T: \RR^d \ni (x_1,\dotsc,x_d) \mapsto \sum_{j \in N(G;T)} x_j - \sum_{i\in T} x_i \]
and we set $\Hc_T = \set{x \in \RR^d \, : \, \sigma_T(x) = 0}$ 
and $\Fc_T = \Pc_G \cap \Hc_T$.

\begin{lemma}[\cite{OH98}]
The facets of $\Pc_G$ are exactly the sets $\Fc_i$ and $\Fc_T$ for all regular vertices $i$ 
and all fundamental sets $T$.
\end{lemma}

A combinatorial condition for a semigroup ring to satisfy Serre's condition $(R_1)$ 
is explicitly stated in \cite[Theorem 2.7]{MV09}. 
In fact, in \cite{MV09} a characterization of $(R_l)$ for all $l$ is given, but for our purposes we only need the case $l=1$.

\begin{proposition}[\cite{MV09}]
\label{BG}
Let $M$ be an affine monoid, $K$ a field and $K[M]$ its semigroup ring.
Then $K[M]$ satisfies Serre's condition $(R_1)$ if and only if 
every facet $\Fc$ of $M$ satisfies the following two conditions:
\begin{enumerate}
\item[(\hspace{0.5mm}i\hspace{0.5mm})]
There exists $x \in M$ such that $\sigma_{\Fc}(x) = 1$, 
where $\sigma_{\Fc}$ is a support form of $\Fc$ taking integer values on $gp(M)$. 
\item[(ii)]
$gp(M\cap \Fc) = gp(M) \cap \Hc$, where $\Hc$ is the supporting hyperplane of $\Fc$; 
\end{enumerate}
where $gp(M)$ denotes the additive group generated by $M$.
\end{proposition}
We apply Proposition \ref{BG} to the affine monoid
\[ 
M_G = \NN (\Pc_G \cap \ZZ^{d}) 
\]
generated by the integer points in $\Pc_G$. 
Note that the support hyperplanes $\Hc_i$ and $\Hc_T$ of $\Pc_G$ 
are also the support hyperplanes of $M_G$. 
We start proving Theorem \ref{main} by the following
\begin{lemma}
Let $G$ be a finite connected nonbipartite graph on the vertex set $[\,d\,]$.
Then the facets of $M_G$ satisfy the first condition of Proposition \ref{BG}.
\end{lemma}
\begin{proof}
First, let $i\in [\,d\,]$ be a regular vertex. 
Since $G$ is connected, there exists an edge $e = \set{i,j} \in E(G)$ to another vertex $j$. 
Then $\sigma_i(\rho(e)) = 1$.

Second, let $T\subset [\,d\,]$ be a fundamental set. 
If $T \cup N(G;T) \subsetneq [\,d\,]$, 
then there exists an edge $e = \set{i,j} \in E(G)$ such that 
$i \in N(G;T)$ and $j \in [\,d\,]\setminus(T\cup N(G;T))$. 
This edge satisfies $\sigma_T(\rho(e)) = 1$. 
If instead $T \cup N(G;T) = [\,d\,]$, 
then every edge of $G$ has either both endpoints in $N(G;T)$, 
or one in $N(G;T)$ and one in $T$. Hence $\sigma_T(e) \in \set{0,2}$ 
for every edge $e$ of $G$. It then follows that 
$\frac{1}{2}\sigma_T$ satisfies the condition of Proposition \ref{BG}.
\end{proof}

To check the second condition of Proposition \ref{BG} we need to compute the lattice generated by $M_G$. 
The following Lemma \ref{lemma:gp} appears in \cite[p. 426]{OH98}
without an explicit proof.  However, 
for the sake of completeness, we give its detailed proof.


\begin{lemma}
\label{lemma:gp}
Let $G$ be a finite connected nonbipartite graph on the vertex set $[\,d\,]$.
Then the lattice $gp(M_G)$ is the set of all integer vectors in $\ZZ^d$ with an even coordinate sum.
\end{lemma}

\begin{proof}
Since every generator of $M_G$ has an even coordinate sum, it follows that
the lattice $gp(M_G)$ is contained in the set of all integer vectors in $\ZZ^d$ 
with an even coordinate sum.

To prove the converse, assume the edges $e_1, \dotsc, e_{\ell}$ 
form an odd cycle of $G$ and let $i$ be the common vertex of $e_1$ and $e_{\ell}$.
Then
\[ 2 \eb_{i} = \sum_{j = 1}^{\ell} (-1)^{j+1} \rho(e_{j}) \, \in gp(M_G).\]
Now consider a spanning tree $G'$ of $G$. The set $\set{\rho(e) \with e \in E(G')}$ together
with $2 \eb_i$ forms a $\ZZ$-basis for the space of all integer vectors in $\ZZ^d$ with an even coordinate sum.
\end{proof}

Now, we can prove two propositions, which complete our proof of Theorem \ref{main}.
\begin{proposition}
\label{Boston}
Let $G$ be a finite connected nonbipartite graph on the vertex set $[\,d\,]$ and
let $i\in [\,d\,]$ be a regular vertex of $G$. Then $\Fc_i$ satisfies the second condition of 
Proposition \ref{BG} if and only if $G_{[\,d\,]\setminus i}$ is connected.
\end{proposition}
\begin{proof}
We denote the connected components of $G_{[\,d\,]\setminus i}$ with $G'_j$. 
Then it is easy to see that $M_{G_{[\,d\,]\setminus i}} = \bigoplus_j M_{G'_j}$ and 
hence $gp(M_G \cap \Fc_i) = gp(M_{G_{[\,d\,]\setminus i}}) = \bigoplus_j gp(M_{G'_j})$. 
Since every $G'_j$ is connected and contains an odd cycle, 
we can use Lemma \ref{lemma:gp} to describe $gp(M_{G'_j})$.
If $G_{[\,d\,]\setminus i}$ is connected, then $gp(G_{[\,d\,]\setminus i})$ and 
$gp(M_G) \cap \Hc_i$ are both the set of integer vectors in $\ZZ^d$ 
with even coordinate sum and $i$th coordinate equal to zero, thus these sets coincide.

We consider the case that $G_{[\,d\,]\setminus i}$ has at least two different connected components 
$G'_1, G'_2$. Then we can choose a vector $x \in \ZZ^d$ such that (i) its coordinate sum 
is even, (ii) $\sigma_i(x) = 0$, and (iii) the restricted coordinate sum over the vertices 
in $G'_1$ is odd. This $x$ is contained in $gp(M_G) \cap \Hc_i$, 
but not in $gp(G_{[\,d\,]\setminus i})$, thus $\Fc_i$ violates the condition.
\end{proof}

\begin{proposition}
Let $G$ be a finite connected nonbipartite graph on the vertex set $[\,d\,]$ and 
let $T \subset [\,d\,]$ be a fundamental set of $G$. Then $\Fc_T$ satisfies the second condition of 
Proposition \ref{BG} if and only if 
one has either $T\cup N(G;T) = [\,d\,]$ or the induced subgraph $G_{[\,d\,]\setminus (T\cup N(G;T))}$ 
is connected.
\end{proposition}

\begin{proof}
Again, we denote the connected components of $G_{[\,d\,]\setminus (T\cup N(G;T))}$ with $G'_j$. We claim that
\[\label{eq:clai}
gp(M_G \cap \Fc_T) = \bigoplus_j gp(M_{G'_j}) 
\oplus \set{x \in \ZZ^d \with \supp(x) \subset T\cup N(G;T), \sigma_T(x)=0}.
\]
Here, $\supp(.)$ denotes the support of a vector. The sum is direct,
because the supports of the summands are disjoint.
$M_G \cap \Fc_T$ (and thus $gp(M_G \cap \Fc_T)$) is generated by the set $\set{\rho(e)\with e\in E(G), \sigma_T(\rho(e)) = 0}$.
For an edge $e \in E(G)$, it holds that $\sigma_T(\rho(e)) = 0$ if and only if either both endpoints lie in $T\cup N(G;T)$ or both are not contained in this set. Thus, a set of generators of the left side of \eqref{eq:clai} is contained in the right side of the equation, and hence one inclusion follows.
Furthermore $\bigoplus_j gp(M_{G'_j}) \subset gp(M_G \cap \Fc_T)$.
Thus it remains to show that
\[ 
\set{x \in \ZZ^d \with \supp(x) \subset T\cup N(G;T), \sigma_T(x)=0} \subset gp(M_G \cap \Fc_T).
\]
For this we consider a spanning tree of the induced bipartite graph on $T \cup N(G;T)$.
Its edges form a $\ZZ$-basis for the left set, hence it is contained in $gp(M_G \cap \Fc_T)$.
Next, we note that
\begin{align*} gp(M_G) \cap \Hc_T &= \set{x\in \ZZ^d \with \supp(x) \cap (T\cup N(G;T)) = \emptyset, 
\sum x_i \text{ even}} \\ &\oplus \set{x\in \ZZ^d \with \supp(x) \subset (T\cup N(G;T)), \sigma_T(x)=0} \,.
\end{align*}
Now the reasoning is completely analogous to the proof of Proposition \ref{Boston}.
\end{proof}

\end{document}